\documentclass{article}
\usepackage{amsfonts,amssymb,mathrsfs,amscd,amsmath,amsthm}
\usepackage[english]{babel}
\usepackage[T2A]{fontenc}
\usepackage[pdftex]{graphicx}
\usepackage{hyperref}
\usepackage{verbatim}
\usepackage{misccorr}
\usepackage{mathtools}
\usepackage{indentfirst}
\usepackage{alltt}
\usepackage{stmaryrd}
\usepackage{url}
\usepackage{enumitem}

\newcommand{\ig}[4]{
    \begin{figure}[!ht]\begin{center}%
    \includegraphics[height = #2\textheight]{pictures//#1}\caption{#4}\label{#3}%
    \end{center}\end{figure}
}
\newtheorem{theorem}{Theorem}
\newtheorem{lemma}[theorem]{Lemma}
\newtheorem{consequence}[theorem]{Consequence}
\newtheorem{comm}[theorem]{Remark}

\newtheorem{sentence}[theorem]{Proposition}

\newtheorem{construction}[theorem]{Construction}
\theoremstyle{definition}
\newtheorem{define}{Definition}
\newtheorem{design}[define]{Notation}

\setlist[enumerate]{label = \rm {(\arabic*)}}

\newcommand{\N}{\mathbb{N}}

\newcommand{\R}{\mathbb{R}}

\renewcommand{\L}{\mathbb{L}}

\newcommand{\w}{\omega}

\newcommand{\sumi}[1]{\underset{i = 1}{\overset{#1}{\sum}}}

\newcommand{\x}{\times}
\newcommand{\va}{\tilde{a}}

\newcommand{\vb}{\tilde{b}}
\newcommand{\vA}{\tilde{A}}

\newcommand{\vB}{\tilde{B}}

\newcommand{\eps}{\varepsilon}
\newcommand{\ca}[1]{\mathcal{#1}}
\newcommand{\smplx}[1]{\triangle_{#1}}
\newcommand{\dis}{\operatorname{dis}}
\newcommand{\dH}{\operatorname{d_H}}
\newcommand{\dGH}{\operatorname{d_{GH}}}
\newcommand{\diam}{\operatorname{diam}}

\newcommand{\s}{\operatorname{s}}
\renewcommand{\t}{\operatorname{t}}
\newcommand{\e}{\operatorname{e}}
\newcommand{\id}{\operatorname{id}}

\newcommand{\neww}{\operatorname{\underline{new}}}
\newcommand{\oldd}{\operatorname{\underline{old}}}
\newcommand{\rightt}{\operatorname{\underline{right}}}
\newcommand{\leftt}{\operatorname{\underline{left}}}
\newcommand{\farr}{\operatorname{\underline{far}}}
\newcommand{\farleftt}{\operatorname{\underline{far~left}}}
\newcommand{\farrightt}{\operatorname{\underline{far~right}}}
\newcommand{\nearestt}{\operatorname{\underline{nearest}}}
\newcommand{\closestt}{\operatorname{\underline{closest}}}

\newcommand{\Leftt}{\operatorname{\underline{Left}}}

\renewcommand{\:}{\colon}

\title{Density of generic metric spaces in the Gromov--Hausdorff class.}
\author{Vikhrov Anton}
\date{}
\begin{document}
    \maketitle
    \begin{abstract}
        In this paper we prove that generic metric spaces are everywhere dense in the proper class of all metric spaces endowed with the Gromov--Hausdorff distance.
    \end{abstract}
    \section{Introduction}
        A symmetric mapping $d\colon X \times X \to [0,\infty]$ that vanishes on the diagonal and satisfies the triangle inequality is called a \emph{generalized pseudometric}. If, in addition, the function is equal to zero only on the diagonal, then it is called \emph{generalized metric}, and if it does not take infinite values, then \emph{metric}.\par
        The Gromov--Hausdorff distance is a value reflecting the degree of difference between two metric spaces. This distance was introduced by Gromov in 1981~\cite{GROMOV} and was defined as the smallest Hausdorff distance between isometric images of given spaces. With the help of this distance, Gromov investigated the properties of groups of polynomial growth. An equivalent definition of this distance was later given.\par
        In this paper, we are using the von Neumann--Bernays--Gödel system of axioms, which introduces the so-called classes and proper classes that generalize the concept of a set. The collection of all metric spaces considered up to isometry is a proper class and is denoted by $\mathcal{GH}$.\par
        It is well known that the Gromov--Hausdorff distance is a generalized pseudometric on $\mathcal{GH}$. In~\cite{ISO} the notion of a space in general position in $\mathcal{GH}$ is introduced and it is shown that such spaces are dense in the space $\mathcal{M}$ of compact non-empty metric spaces considered up to isometry, and the structure of small neighborhoods of a generic space in $\mathcal{GH}$ is studied. These facts imply the triviality of the isometry group of the space $\mathcal{M}$. In this paper we prove that generic spaces is everywhere dense subfamily in $\mathcal{GH}$.\par
        The author expresses his gratitude to his supervisor, Dr. Sci. Professor A.A.Tuzhilin, as well as Dr. Sci. Professor A.O.Ivanov for posing the problem and attention to the work.
    \section{Main definitions and preliminary results}
First we introduce some basic notation. We denote by $\R_{\geq 0}$ the set of non-negative real numbers, and by $\R_{+}$ the set of positive real numbers.
Let $(X, \rho)$ be an arbitrary metric space, and $x,y \in X$. The distance between the points $x$ and $y$ is denoted by $|xy| = \rho(x,y) = d_X(x,y)$. Let $U_{\eps}(a)$ be an open ball with center $a$ of radius $\eps$, and
$U_{\eps}(A)~ = ~\underset{a \in A}{\bigcup}~U_{\eps}(a)$ be a~$\eps$-neighborhood of a non-empty subset $ A$, and $S_{\eps}(a)$ is a sphere of radius $\eps$ centered at the point~$a$.
We denote by $\#X$ the cardinality of $X$, and for any $a \in \R_{\geq0}$ and metric space $X$ we put $a\,X = (X, a\,d_{X})$.\\
\begin{define}
Let $A,B$ be non-empty subsets of a metric space $X$. \emph{The Hausdorff distance} is the value
\begin{equation*}
\dH(A,B) = \underset{r \in \R_{\geq 0}}{\inf} \Bigl\{r\: A \subseteq U_r(B)\, \&\, B \subseteq U_r (A)\Bigr\}.
\end{equation*}
\end{define}

\begin{define}
Let $A,B,X$ be metric spaces. If $A$ is isometric to $\vA$ and $B$ is isometric to $\vB$, where $\vA$ and $\vB$ are subspaces of $X$, then the triple $(\vA,\vB,X) $ we call \emph{realization of the pair} $(A,B)$.
\end{define}

\begin{define}
\emph{The Gromov--Hausdorff distance} between two metric spaces $A$, $B$ is the infimum of the Hausdorff distances among all realizations of the pair $(A,B)$. In other words,\\
\begin{equation*}
\dGH(A,B) = \inf\bigl\{r: \text{there is a realization } (\vA,\vB,X) \text{ of the pair ($A,B$) such that }d_{H} (\vA,\vB) \leq r \bigr\}.
\end{equation*}
\end{define}

\begin{define}
\emph{A correspondence} between two sets $A$ and $B$ is a subset $R \subseteq A \x B$ such that for any $a \in A$ and $b \in B$ there exist $\va \in A$ and $\vb \in B$ for which $(a,\vb)$, $(\va, b)$ belong to $R$.
\end{define}
Further, $aRb$ means that $ a $ and $b$ are in correspondence $R$, and the set of all correspondences between metric spaces $A$, $B$ is denoted as $\ca{R}(A,B) $.

\begin{define}
Let $R$ be a correspondence between metric spaces $A$, $B$. Its \emph{distortion} is given by
\begin{equation*}
\dis R = \sup \Bigl\{\bigl|d_X(a,a')-d_Y(b,b')\bigr|:\, a R b\, \text{ and } \,a' R b '\Bigr\}.
\end{equation*}
\end{define}
\begin{comm}
If $R$ is a functional correspondence, that is, there exists a mapping $f: X \to Y$ such that $xRy$ if and only if $ y = f(x)$, then its distortion can be written in the form
\begin{equation*}
\dis f = \sup \Bigl\{\bigl|d_X(a,a')-d_Y(f(a),f(a'))\bigr|,\, a,a' \in X \Bigr\}.
\end{equation*}
\end{comm}
\begin{sentence}[\cite{BUR}]\label{theorem:main_formula}
For any metric spaces $A$ and $B$, the following equality holds\/$\:$
\begin{equation*}
2\dGH(A,B) = \underset{R \in \ca{R}(A,B)}{\inf} \dis R.
\end{equation*}
\end{sentence}
\begin{define}
\emph{The Gromov--Hausdorff class} $\mathcal{GH}$ is a proper class (in terms of von~Neumann--Bernays--Gödel set theory) of all metric spaces considered up to isometry.
\end{define}~\begin{sentence}[\cite{BUR}]
The Gromov--Hausdorff distance is a generalized pseudometric on $\mathcal{GH}$.
\end{sentence}
Denote by $\smplx{n}$ an $n$-point \emph{simplex}, that is, a metric space of cardinality $n$, such that the distances between  its different points are equal to $1$.
The diameter of a metric space $X$ is
\begin{equation*}
\diam(X) = \underset{x, x' \in X}{\sup}d_X(x,x').
\end{equation*}~\begin{sentence}[\cite{BUR}]~\label{lemma:dist_to_smplx}
For any metric space $X$, the formula $2\dGH(\smplx{1},X) = \diam(X)$ is valid.
\end{sentence}
\begin{design}
Let $X \in \mathcal{GH}$. Denote by $S(X)$ the set of all bijective mappings of $X$ onto itself. We put
\begin{gather*}
\s(X) = \inf\bigr\{|xx'|: x \neq x'; x,x' \in X\bigl\},\\
\t(X) = \inf\bigr\{|xx'|+|x'x''|-|xx''|: x \neq x' \neq x'' \neq x; x,x',x'' \in X\bigl\},\\
\e(X) = \inf\bigr\{\dis(f), f \in S(X), f \neq \id\bigl\}.
\end{gather*}
\end{design}
\begin{define}
A metric space $X$ is called \emph{a generic space} if all three quantities $\s(X)$, $ \e(X)$, $\t(X)$ are positive.
\end{define}
    \subsection{Canonical projection}
Recall how to construct a pseudometric space from a connected weighted graph. Everywhere below, graphs are assumed to be simple, connected, and weighted, and the edge weight function (given on the edges of the graph) is non-negative. The sets of vertices of the graphs and the sets of edges can be infinite. The vertices of the graphs are sometimes called their vertices. The edge connecting $x$ and $y$ is denoted by $xy$ or $x \sim y$.
\begin{define}
\emph{A generalized walk} $L$ in the a graph $G$ connecting its points $x$ and $y$ is a finite sequence $\bigl[x_i, i = 1 \dots N\mid x_1 = x, x_N = y\bigr]$, $N \geq 2$, such that either $x_i x_{i+1}$ is an edge for all $i$, or $x_i = x_{i+1}$.\footnote{We use square brackets in order to make this text easier to read.} \emph{An edge of a generalized walk} is an edge connecting successive distinct points of this walk. \emph{The length of the walk} $L$ is defined as $|L|$ = $\sumi{N-1}\w(x_i x_{i+1})$. \emph{The set of generalized walks} connecting $x$ and $y$ is denoted by $\L(x,y)$.
\end{define}
In this paper, \textbf{a generalized walk is simply called a walk}
\begin{define}\label{def:kanon_proj}
Let us call the mapping $\pi$, which assigns to each connected weighted graph $X = (V,E, \w)$ the metric space $(V,d_{\w})$, \emph{canonical}, where
\begin{equation*}
d_{\w}(y_1,y_2) = \begin{cases}
\inf\bigl\{|L|, L \in \L(y_1,y_2) \bigr\}& \text{for } y_1\neq y_2,\\
0& \text{for } y_1 = y_2.
\end{cases}
\end{equation*}
We say that such projection \emph{preserves the edge weights} if $\w(xy) = d_{\w}(x,y)$ is true for any $xy \in E$.
\end{define}
It is well known that $(Y, d_{\w})$ is a pseudometric space.
\begin{comm}\label{comm:is_metric_space}
Let $X = (V, E, \w)$ be a graph. If there exists $ C > 0$ such that $\w(e) \geq C$ for all $\e \in E$, then $\pi(X)$ is a metric space.
\end{comm}
\begin{define}
Let $X$ be a weighted graph and $z_1 z_2$ its edge. Then \emph{the polygon inequality for the lower base} $z_1 z_2$ and the walk $L \in \L(z_1,z_2)$ is the inequality $\w(z_1 z_2) \leq |L|$.
\end{define}
\begin{lemma}\label{lemma:sohrani}
The canonical projection preserves edge weights if and only if all polygon inequalities hold for all lower bases $xy$ $\in$ $E$ and any $L \in \L(x,y)$.
\end{lemma}
\begin{proof}
Note that $\inf \bigl\{|L|: L \in \L(x,y) \bigr\} \leq \w(xy)$ because $xy \in \L(x, y)$. Due to the polygon inequality with the lower base $xy$, we have $\w(xy) \leq |L|$ for every $L \in \L(x,y)$. \par
If the polygon inequality does not hold for at least one pair of points $x,y$, i.e., there is a walk $L \in \L(x,y)$ such that $\w(xy) > |L|$, then $d_{\w}(x,y) < \w(xy)$.
\end{proof}

    \subsection{Subdivision of a metric space.}
Let us generalize the notion of graph subdivision.
\begin{construction}~\label{constr:podr}
Consider an arbitrary metric space $X$ as a complete weighted graph $G'$ with the weight function $\w$ equal to the distance between the points. For each pair of points $\{$$u$, $v$$\}$, add the points $\alpha^{u,v}_i$, $i \in \mathcal{I}(u,v)$ to the graph $G'$. Connect each $\alpha^{u,v}_i$ to each $\alpha^{u,v}_j$, and connect the points $u$, $v$ to all $\alpha^{u,v}_i$. To the added edges we assign arbitrarily weights in such a way that the triangle inequalities hold in all the subgraphs $G_{u,v}$ generated by $\{u,v,\alpha^{u,v}_{i},\/i \in I(u,v)\}$.  We denote the obtained graph by $G$.
\end{construction}
We put $Z = \pi(G)$, and the points obtained from $X$ is denoted in the same way as in $X$. Let us write some properties of the space $Z$.
\begin{lemma}~\label{lemma:podr}
\begin{enumerate}
\item The projection $\pi$ preserves the weights of all edges.
\item The distance between points $x,y$ located in $G_{u,v}$ and $G_{u',v'}$, respectively, where $uv \neq u'v'$ and $x$, $ y \notin X$, is equal to the minimal length of the following walks$\:$
\begin{enumerate}
\item $L_1$ $ = $ $[x,u,u',y]$,
\item $L_2$ $ = $ $[x,u,v',y]$,
\item $L_3$ $ = $ $[x,v,u',y]$,
\item $L_4$ $ = $ $[x,v,v',y]$.
\end{enumerate}
\item The distance from a vertex x of $G_{u,v}$, where $v \neq x \neq u$, to $u' \in X$, where $v \neq u' \neq u$ is equal to the minimal length of the following walks$\:$
\begin{enumerate}
\item $L_1$ $ = $ $[x,u,u']$,
\item $L_2$ $ = $ $[x,v,u']$.
\end{enumerate}
\end{enumerate}
\end{lemma}
\begin{proof}
Let us verify $(1)$. To prove it, we need to check all polygon inequalities and apply Lemma~\ref{lemma:sohrani}.\par
Consider a pair of points $u,v$, which were obtained from the space $X$. Any walk connecting them can be divided into segments $L_{i,j}$ that lie entirely in $G_{x_i,x_j}$ for some $x_i, x_j \in X$, and neighboring walks must intersect at points from $X$. The length of each such segment is no less than the one of $[x_i,x_j]$, which means that the infimum can be calculated by walks passing only through points from $X$, and each such walk is no shorter than $[u,v]$.\par
Consider an edge $xy$ that lies in $G_{u,v}$ and is distinct from $uv$. The length of any walk connecting $x$, $y$, which lies entirely in $G_{u,v}$, is no less than $ \w(xy)$ due to polygon inequalities. If the walk passes through the point $u$ and exits $G_{u,v}$, then it must pass through the point $v$, since the walk does not pass through the same point $u$ twice, and the graph $G \backslash \{u,v\}$ (subgraph $G$ spanned by all vertices of $G$ except $u,v$) is disconnected (and the walk connects points lying in different connected components). Any walk connecting $u,v$ is no shorter than $ [u,v]$ (see the proof of the first item). Hence, a walk that does not lie in $G_{u,v}$ is no shorter than some walk lying entirely in $G_{u,v}$, and this case was considered at the beginning of the proof.\par
Now let us prove item $(2)$. Consider an arbitrary walk $L$ connecting $x,y$, and let its first vertex in $X$ be $a_1$ and the last one be $ a_2 $(the walk must pass through points from $X$, since $G\backslash X$ is not a connected graph, and the points $x$, $y$ lie in different connected components). Any walk connecting $a_1$ and $a_2$ is no shorter than $ [a_1,a_2]$. Any walk that connects $x, a_1$ is no shorter than $ [x,a_1]$, similarly with $y$ and $a_2$ due to polygon inequalities. That is, the distance is calculated at least from the walk s $[x,a_1,a_2,y]$. The point $a_1$ must lie in the same connected component as $x$, and $a_2$ --- as $y$. Item proven.\par
Finally, we prove item $(3)$. An arbitrary walk connecting $ x$ to $u'$ must pass through some $a \in \{u,v\}$, since $x$ and $u'$ lie in different connected components of $G \backslash \{u,v\}$. An arbitrary walk connecting $x,a$ is not shorter than $[x,a]$, and $a$ and $u'$ are not shorter than $[a,u']$. The lemma is proved.
\end{proof}
We call such a construction \emph{a subdivision of the metric space X}. In this article, this construction is used in its simplest form : $\# I = 1$.
        \subsection{Metrically convex functions.}
\begin{define}
Non-constant function $f\: \R_{\geq 0} \to \R_{\geq 0}$ for which $f(0) = 0$ and for any $a \leq b + c$, $a, b, c \geq 0$, the inequality $f(a) \leq f(b) + f(c)$ holds, we call \emph{metrically convex}.
\end{define}
The simplest properties of metrically convex functions immediately follow from the definition. Since the inequality $a \leq b + 0$ implies $f(a) \leq f(b)$, we obtain the following result.
\begin{sentence}
A metrically convex function is non-decreasing.
\end{sentence}
Applying the definition twice, we get
\begin{sentence}
The composition of metrically convex functions is metrically convex.
\end{sentence}
\begin{sentence}
A metrically convex function vanishes nowhere but zero.
\end{sentence}
\begin{proof}
Indeed, if $f(a) = 0$ for $a > 0$, then $f(x) = 0$ for $x \leq a$. Let $a_0~ = ~\sup~\{~x:~f(x)~ = ~0~\}$. Then for all $y > a_0$, we get $f(y) > 0$, but for $ a_0 < y = 2a_0/3 + 2a_0/3$, we have $0 < f(y) < f(2a_0/3) + f(2a_0/3) = 0$, a contradiction.
\end{proof}
\begin{consequence}
For any metric space $(X, d_X)$ and a metrically convex function $f$, the pair $(X,f \circ d_X)$ is a metric.
\end{consequence}
By $f(X)$ denote the metric space $(X, f \circ d_X)$. For a non-empty subset $A \subseteq \R_{\geq 0}$, we put
    \begin{equation*}
        ||f(x)||_{A} = \sup \bigl\{|f(x)|\, \: x \in A\}.
    \end{equation*}
\begin{sentence}~\label{sent: is_metric_open}
A nondecreasing function $f$ equal to zero at zero, is metrically convex if and only if for arbitrary $b,c > 0$ the inequality
\begin{equation*}
f(b+c) \leq f(b) + f(c)
\end{equation*}
    holds.
\end{sentence}
\begin{proof}
If $f$ is metrically convex, then the inequality holds by definition (it suffices to put $a = b+c$).\par
Conversely, if $0 < a \leq b + c$, then $f(a) \leq f(b+c) \leq f(b) + f(c)$. If at least one of the numbers $a,b,c$ is equal to zero, then the inequality is satisfied due to non-decreasing or non-negativity of the function.
\end{proof}
\begin{lemma}~\label{lemma: dist_to_image}
For any metric space $X$ and metrically convex function $f$, the inequality
\begin{equation*}
2 \dGH \bigl(X,f(X)\bigr) \leq ||x - f(x)||_A
\end{equation*}  holds, where $A = \bigl[\s(X),\diam(X)\bigr]$ if $\diam(X) < \infty$ and $A = \bigl[\s (X), \infty\bigr)$ otherwise.
\end{lemma}
\begin{proof}
It suffices to consider $R = \bigl\{(u,u) \mid u \in X\bigr\}$, for which $\dis(R) \leq ||x - f(x)||_{A}$, where $A$ is the set from the hypothesis of the theorem. Hence $2\dGH(X,f(X)) \leq \dis(R) \leq ||x - f(x)||_{A}$.
\end{proof}
\newpage
        \section{Density of the family of generic metric spaces in $\ca{GH\/}$.}
\begin{define}
Let us call the \emph{$\eps$-ladder} the following function
\begin{equation*}
l_{\eps}(x) = \begin{cases}
k\eps,&x \in \bigl((k-1)\eps, k\eps\bigr];\, k = 1 \dots \infty,\\
0,&\, x = 0.
\end{cases}
\end{equation*}
\end{define}
\ig{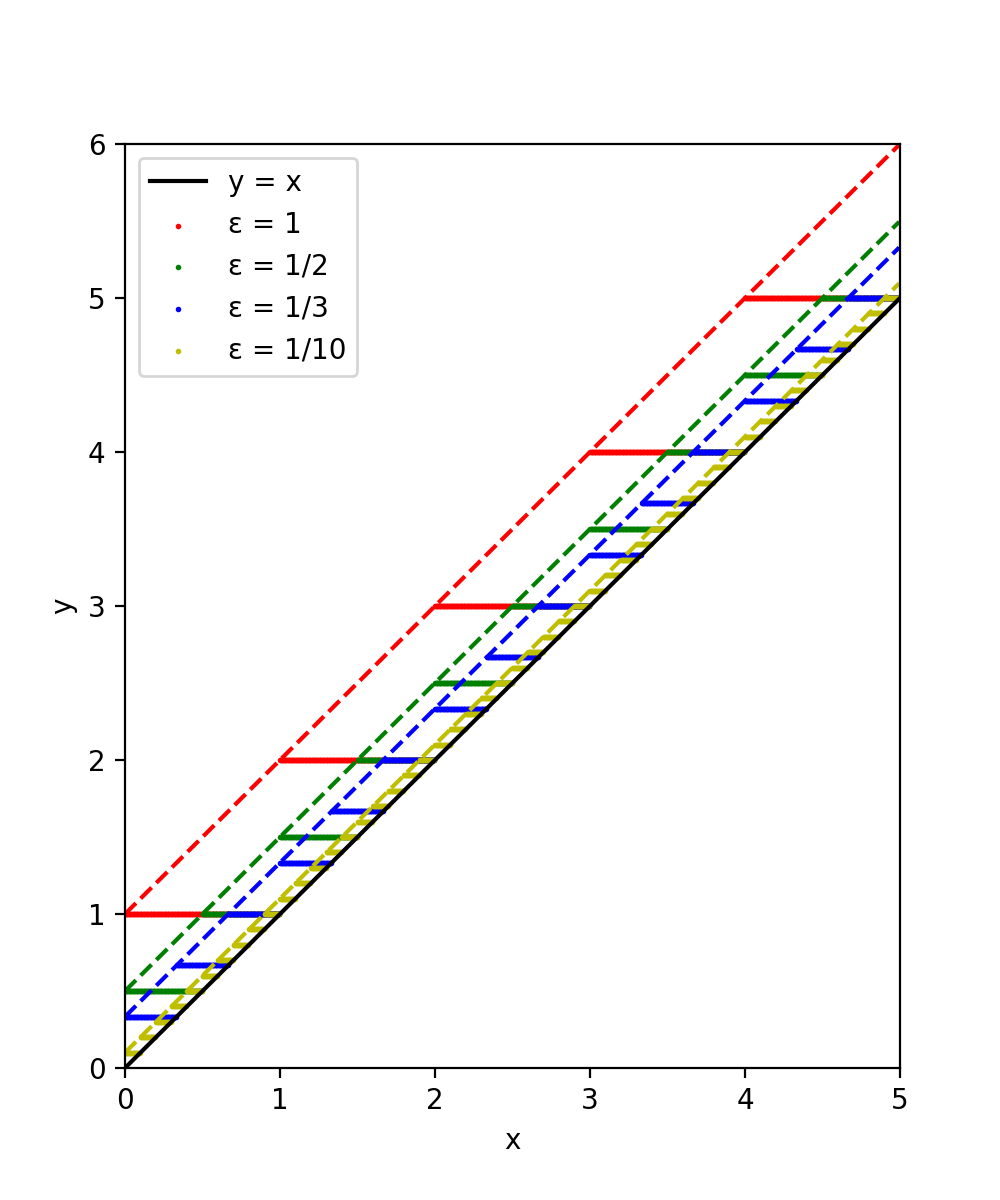}{0.3}{fig:l(x)}{Functions $l_{\eps}(x)$.}
\begin{lemma}
The $\eps$-ladder is metrically convex.
\end{lemma}
\begin{proof}
Let us assume that $\eps$ = $1$, since other cases are obtained by composition with a linear function and the composition of metrically convex ones is metrically convex. By the proposition~\ref{sent: is_metric_open}, it suffices to check the inequality $l_{1}(a+b) \leq l_{1}(a) + l_{1}(b)$ for $a,b > 0 $. Let us represent $a,b$ as $a = A - \alpha$, $b = B - \beta$, where $A,B \in \N$ and $\alpha$ and $\beta$ are non-negative and smaller then 1. By definition, $l_{1}(a) = A$, $l_{1}(b) = B$, and
\begin{equation*}
l_{1}(a+b) = l_{1}(A + B - \alpha - \beta) \leq l_{1}(A+B) = A + B = l_{1}(a) + l_ {1}(b)
\end{equation*}
due to non-decreasing.
\end{proof}
\begin{design}
Let $X$ be a metric space and $c > 0$. By $X + c$ denote the result of applying the function
\begin{equation*}
f(x) = \begin{cases}
x + c,& x > 0,\\
0,& x = 0,
\end{cases}
\end{equation*}
to the metric space $X$.
\end{design}
\begin{comm}
Such a function f is indeed metrically convex because for $x,y,z > 0,$ $x \leq y + z$, the inequalities $x+c \leq y +c + z + c$ hold, and if $ y = 0$, then $f(x) \leq f(y)$ due to monotonicity of the function f.
\end{comm}
\begin{sentence}\label{lemma:xplusc}
For a metric space $X$ and $c >0$,
\begin{enumerate}
\item $\s(X+c)$ = $\s(X)+c$,
\item $\t(X+c)$ = $\t(X)+c$,
\item $\e(X+c)$ = $\e(X)$.
\end{enumerate}
\end{sentence}~\begin{lemma}[\cite{AMERICAN}]~\label{lemma:sohr_poradok}
Let  $(X,\prec)$ be a well-ordered set and $\phi \: X \to X$ be some order-preserving bijection. Then $\phi$ is the identity mapping.
\end{lemma}
\begin{theorem}~\label{lemma:base}
For any non-negative $\delta$ and $c$, in a $(\delta + c)$-neighbourhood of a metric space X $\in \ca{GH}$ there exists a generic metric space $U$ such that $\s(U) \geq \delta/3 + 2c$, $\t(U) \geq 2c$, and $\e(U) \geq \delta/3$.
\end{theorem}
\begin{proof}
Let us prove the theorem for the case $X = \smplx{1}$. Consider a three-point metric space $U = \bigl(\{u_1,u_2,u_3\}, d_u\bigr)$, where \par
\begin{enumerate}
\item$d(u_1,u_2) = \delta/3 + 2c$,
\item$d(u_2,u_3) = 2\delta/3 + 2c$,
\item$d(u_1,u_3) = 3\delta/3 + 2c$.
\end{enumerate}
For such a metric, we have $s(U) = 2c + \delta/3$, $t(U) = 2c$, $e(U) \geq \delta/3$, because all distances in the space $U$ differ minimum by $\delta/3$. Further we will consider the case $\#X$ > 1.\par
First, we apply the simplest $\eps$-ladder to the original space for an arbitrary $\eps > 0$. Then by Lemma~\ref{lemma: dist_to_image},
\begin{equation*}
    2\dGH\bigl(X,l_{\eps}(X)\bigr) \leq \|l_{\eps}(x)-x\|_{\R_{\geq0}} \leq \eps.
\end{equation*}\par
Let $Z' = l_{\eps}(X)$. We fix some complete order on the set $Z'$ (such an order exists by Zermelo's theorem) and denote it by $\prec$. Let us subdivide the metric space $Z'$ as follows : for each ordered pair of points $(z_1,z_2)$ from $Z' \x Z'$, where $z_1 \prec z_2$, we add a new point $z_3$, which we connect by edges with $z_1$ and $z_2$ and define the weights of these edges as follows : $\w(z_3 z_2) = \eps/4,$ $ \w(z_1 z_3) = d_{Z'}(z_1,z_2)-\eps /4$. The subgraph $G_{z_1,z_2}$ satisfies the triangle inequalities.\par
\ig{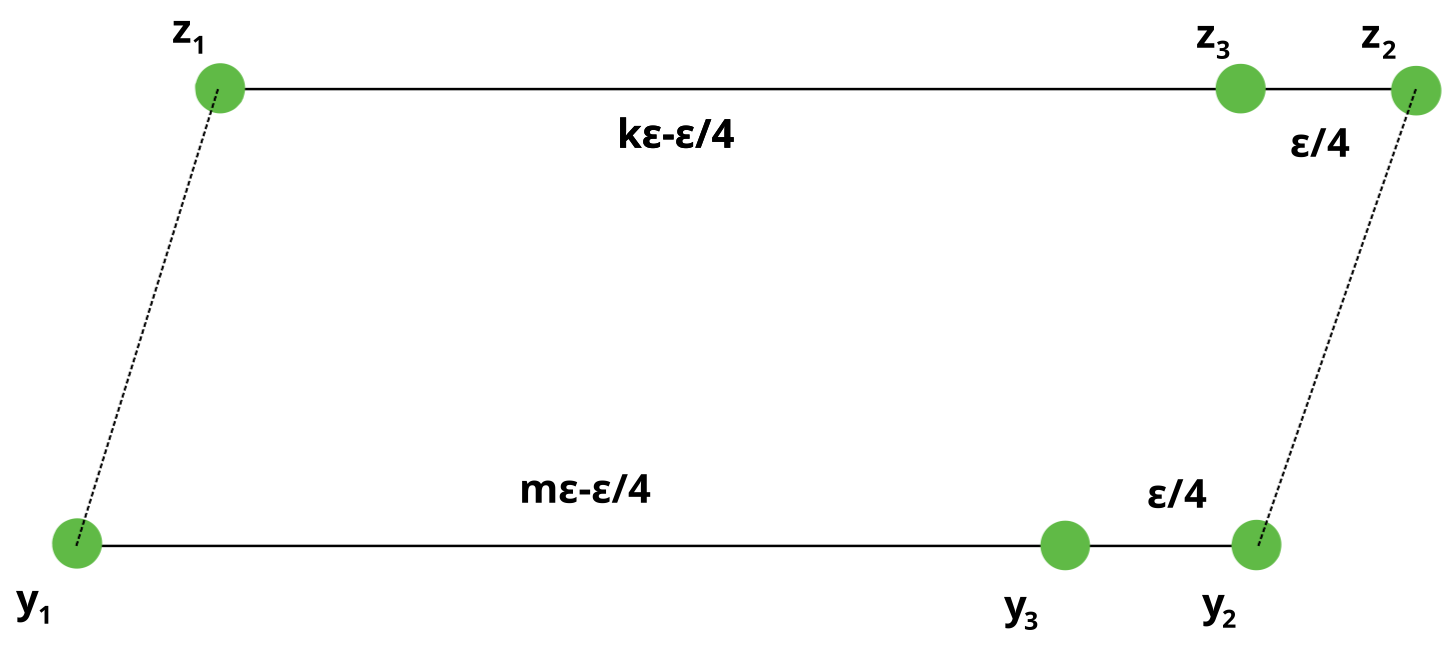}{0.2}{fig:pic1}{Part of graph $G$, $z_3$ added for $z_1 \prec z_2$ pair and $y_3$ for $y_1 \prec y_2$.}
Let us introduce the following terminology.
\begin{enumerate}
\item Points of the graph $G$ lying in $Z'$ is called \emph{$\oldd$}.
\item Remaining points of the graph $G$ is called \emph{$\neww$}.
\end{enumerate}
Let $z_3$ be added for the pair $z_1 \prec z_2$.
\begin{enumerate}[resume]
\item Note that the points $z_1$ and $z_2$ are the only points of the graph $G$ that are connected by an edge to $z_3$, and $\w(z_1 z_3) = k\eps - \eps/4$, and $ \w(z_3 z_2)$ is equal to $\eps/4$. Let us call the point $z_2$ as \emph{$\rightt$}($z_3$) and $z_1$ as \emph{$\leftt$}($z_3$). We denote by \emph{$\nearestt$}($z_3$) a variable that can take value $\rightt$$(z_3)$ or $\leftt$$(z_3)$.
\item The $\oldd$ point, which is not the $\nearestt$$(z_3)$, is denoted as \emph{$\farr$}($z_3$).
\end{enumerate}\par
Here and below, $z_3$ and $y_3$ denote $\neww$ points, $z_1$ = $\leftt$($z_3$), $y_1$ = $\leftt$($y_3$), $z_2$ = $\rightt$($z_3$), $y_2$ = $\rightt$($y_3$). The $\neww$~$\sim$~$\rightt$ denotes an unordered pair of distinct point types described above, consisting of some $\neww$ $z_3$ and a $\rightt$$(z_3)$; $\neww$ $\sim$ $\neww$ denotes a pair of $\neww$ and $\neww$ points; $z_3$~$\sim$~$\farr$($z_3$) denotes a pair of $\{$$\neww$, $\farr$($\neww$)$\}$, and so on.\par
We put $Z = \pi(G)$ and continue to call the image of $Z'$ as $Z'$. By Lemma~\ref{lemma:podr}, the projection $\pi$ has preserved the distances. By Remark~\ref{comm:is_metric_space}, the space $Z$ is metric. We will also consider $Z$ as a weighted graph, where the weight function is the distance. Due to the fulfillment of the polygon inequalities all edge weights are preserved in the space $Z$ (see Lemma~\ref{lemma:sohrani}). Let us describe what other distances look like in the new space.\par
The $\pi$ projection preserves the weights of the edges connecting the $\oldd$ points with themselves and connecting the $\neww$ ones with $\closestt$ to them. Thus, $d(z_3,z_2) = \eps/4$, $d(z_1,z_3) = k\eps - \eps/4$, $d(z_1,z_3) = k\eps$ for some natural k, and for any $u,v \in Z'$, $d(u,v) = m\eps$ for some $m \in \N$.
If the distance between points is equal to $k\eps + \alpha$ for $k\in \N\cup{0}$ and $0\leq \alpha < \eps$, then we say that the distance between points has the form $k\eps + \alpha$ or $k\eps - (\eps-\alpha)$. Let us show how the remaining distances in $Z$ are structured.
\begin{lemma}
The distance from the $\neww$ $z_3$ to $y = $$\farr$\,$(z_3)$ is equal to the length of the walk
\begin{center}
$\bigl[$$\neww$ $z_3$, $\closestt$$(z_3)$, $\oldd$ $y$$\bigl]$,
\end{center}
and can be written in the form
\begin{enumerate}
\item $k\eps - \eps/4$ or
\item $k\eps + \eps/4$
\end{enumerate}
for some non-negative integer $k$.
\end{lemma}
\begin{proof}
Due to point $3$ of the Lemma~\ref{lemma:podr}, the shortest walk $L$ connecting $z_3$ and $y$ looks like
\begin{center}
$L$ = $\bigl[$$\neww$ $z_3$, $\nearestt$($z_3$), $\oldd$ $y$$\bigr]$,
\end{center}
and, due to the structure of the weights of the edges of the graph $G$, its length is equal to
\begin{enumerate}
\item $k\eps -\eps/4 + m\eps$ for some $k, m \in \N$ if $L$ goes through $z_1 = $ $\leftt$$(z_3)$, or
\item $\eps/4 + m\eps$ for some $m \in \N$ if $L$ goes through $z_2 = $ $\rightt$$(z_3)$.
\end{enumerate}
\end{proof}
\begin{lemma}
The distance from the $\neww$ $z_3$ to the $\neww$ $y_3$ is equal to the length of three-edge walk $L$
\begin{center}
$\bigl[$$\neww$ $z_3$, $\nearestt$$(z_3)$, $\nearestt$$(y_3)$, $\neww$ $y_3$$\bigr]$,
\end{center}
and can be written in form
\begin{enumerate}
\item $k\eps-\eps/2$ or
\item $k\eps+\eps/2$ or
\item $k\eps$
\end{enumerate}
for some non-negative integer $k$.
\end{lemma}
\begin{proof}
By point $2$ of the Lemma~\ref{lemma:podr} the shortest walk $L$ is $\bigl[$$\neww$ $z_3$, $u$ = $\nearestt$$(z_3)$, $v$ = $\nearestt$$(y_3)$, $\neww$ $y_3$$\bigr]$ (the second and third points of the walk may coincide), and its length is equal to
\begin{enumerate}
\item $k\eps-\eps/4 + m\eps + p\eps - \eps/4$ if $u$ = $\leftt$($z_3$) and $v$  = $\leftt$($y_3$),
\item $\eps/4 + m\eps + \eps/4$ if $u$ = $\rightt$($z_3$) and $v$ = $\rightt$($y_3$),
\item $k\eps-\eps/4 + m\eps + \eps/4$ or $\eps/4 + m\eps + p\eps -\eps/4$ if one of u, v is $\rightt$ and the other is $\leftt$,
\end{enumerate}
for some positive integers $k,p$ and a non-negative integer $m$.
\end{proof}
Let us divide all pairs of different points $Z$ into 7 classes according to the type of distances between them$\:$
\begin{enumerate}
\item $\eps/4$: $\bigl($$\neww$ $\sim$ $\rightt$$\bigr)$,
\item $k\eps$, $k\geq 1$: ($\oldd$ $\sim$ $\oldd$),
\item $k\eps - \eps/4$, $k\geq 1$: $\bigl($$\neww$ $\sim$ $\leftt$($\neww$)$\bigr)$,
\item\label{itm:old_right} $k\eps + \eps/4$, $k \geq 1$: $\bigl($$\neww$ $\sim$ $\farr$ ($\neww$)$\bigr)$,
\item\label{itm:old-left} $k\eps - \eps/4$, $k \geq 2$: $\bigl($$\neww$ $\sim$ $\farr$ ($\neww$)$\bigr)$,
\item $k\eps$, $k \geq 1$: ($\neww$ $\sim$ $\neww$),
\item $k\eps - \eps/2$, $k \geq 1$: ($\neww$ $\sim$ $\neww$).
\end{enumerate}\par
If the distance from the $\oldd$ point $y$ to the $\neww$ point $z_3$ lies in the class~\ref{itm:old-left} $\bigl($in the class~\ref{itm:old_right}$\bigr)$, then such a point $y$ is called \emph{$\farleftt$}($z_3$) $\bigl($\emph{$\farrightt$}($z_3$)$\bigr)$.\par
Consider an arbitrary bijection $\phi\: Z \rightarrow Z$ with $\dis(\phi) <\eps/4$, then $\dis(\phi^{-1}) < \eps/4$. Let us prove that $\phi = \id$.\par Note that for any $k,m,p,l \in \N$, numbers $k\eps$, $m\eps -\eps/4$, $p\eps + \eps/4$, $l\eps - \eps/2$ differ by at least $\eps/4$, so $\phi$ is an isometry.
\begin{lemma}~\label{sent:class}
If the distance between points is equal to $\eps/4$, that is, the distance is in the first class, then the distance between the images of the mapping $\phi$ can only be from the first class. Similarly, an unordered pair of points at a distance from the class $ (3)$ or $(5)$ can only go to a pair of points at a distance from the class $(3)$ or class $(5)$ due to the difference between distances of classes.\par
\end{lemma}
\begin{lemma}~\label{sent:blizh_left}
The distance from the $\neww$ $z_3$ to the $\leftt$$(z_3)$ is strictly less than the distance from $z_3$ to $y$ = $\farleftt$$(z_3)$ for each $y$.
\end{lemma}
\begin{proof}
Indeed, the distance from $z_3$ to an arbitrary $\farleftt$($z_3$) is calculated along a two-edge walk passing through $z_1$, that is, will be greater than the distance from $z_3$ to $z_1$.
\end{proof}
Thus, the $\leftt$$(z_3)$ is the closest to $z_3$ among all points that are at a distance of the form $k\eps - \eps/4$ from $z_3$.
\begin{lemma}~\label{sent:kol_blizh}
If a point $x$ has $n$ points at the distance $\eps/4$ $($$n$ is a cardinal number$)$, then $\phi(x)$ also has exactly $n$ points at distance $\eps/4$.
\end{lemma}
\begin{proof}
Indeed, there cannot be less points because the images of points that are at a distance $\eps/4$ from $x$, are at a distance $\eps/4$ from $\phi(x)$. Considering the inverse mapping, we obtain the equality.
\end{proof}
\begin{lemma}
The $\neww$ point $z_3$ goes to some $\neww$ $y_3$, and the $\rightt$$(z_3)$ = $z_2$ goes to the $\rightt$$(y_3)$ = $y_2$.
\end{lemma}
\begin{proof}
By Lemma~\ref{sent:class}, the unordered pair of points $\{$$\neww$, $\rightt$$\}$ goes into the unordered pair $\{$$\neww$, $\rightt$$\}$. Suppose $\phi(z_3)$ = $y_2$ = $\rightt$($y_3 $), and $z_2$ = $\rightt$$(z_3)$ and $\phi(z_2)$ = $y_3$. \par
\ig{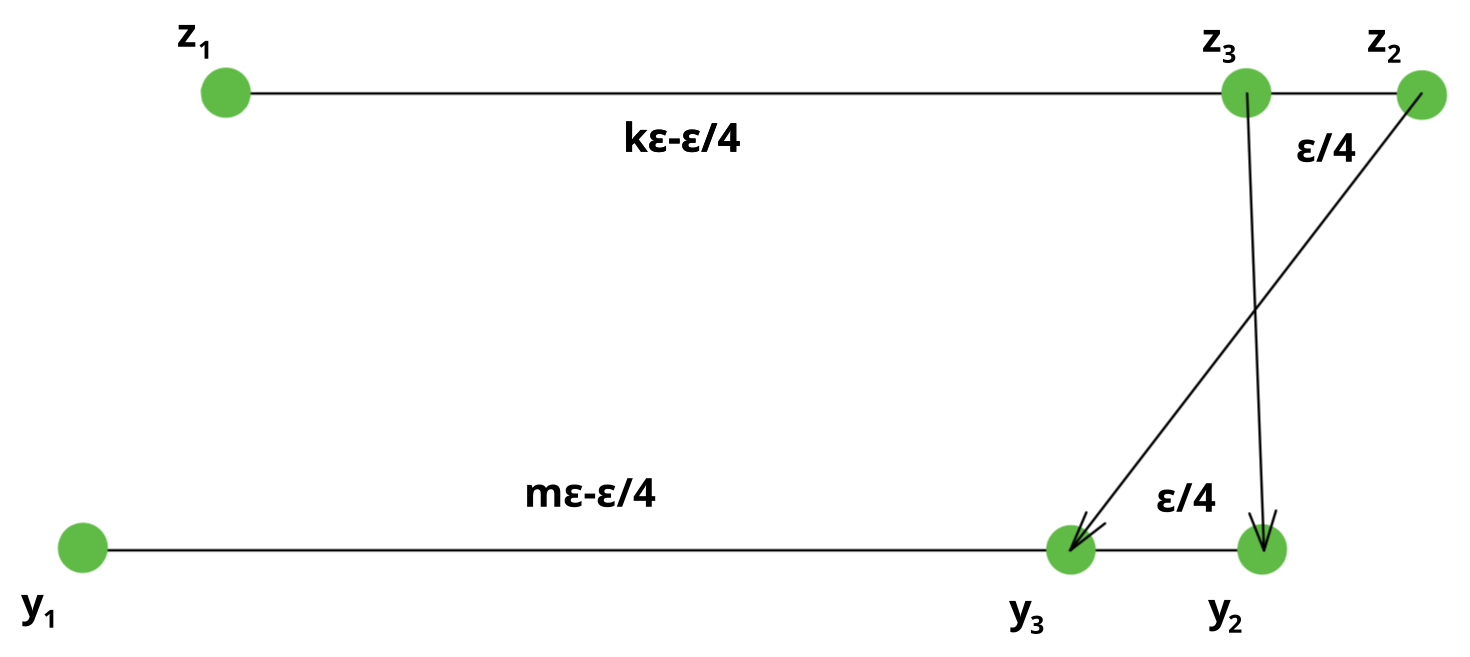}{0.2}{fig:pic3}{$\neww$ $z_3$ goes into $\oldd$ $y_2 = $ $\rightt$$(y_3)$, and $z_2$ = $\rightt$($z_3$) goes into $\neww$ $y_3$.}
The least element $a_1$ of a well-ordered set $Z'$ is called the first, and the least element of the ordered set $Z'\backslash \{a_1\}$ is called the second element $a_2$.\par
Due to the fact that $a_1$ is the only point of $Z$ that has no points at distance $\eps/4$, then by the~Lemma~\ref{sent:kol_blizh}, the point $a_1$ goes to itself. Point $z_3$ has exactly one point at distance $\eps/4$, so $ y_2 = \phi(z_3)$ has exactly one point at distance $\eps/4$ due to the Lemma~\ref{sent:kol_blizh}, i.e. exactly one $\neww$ point for which $ y_2$ is $\rightt$, so exactly one $z \in Z'$ such that $ z \prec y_2$, i.e.\, $y_2$ is second element of the set $Z'$, and $y_1$ is the first one. Considering the inverse mapping, we obtain that $y_1 = a_1 = z_1$, and $z_2 = a_2 = y_2$. Since $d(z_1,z_3) = k\eps - \eps/4$ due to the fact that $z_1 z_3$ is a $\leftt$ - $\neww$ pair, and from the above, $d\bigl(\phi(z_1), \phi(z_3)\bigr) = d(y_1,y_2) = k\eps$, a contradiction.\par
\end{proof}
We get that any $\neww$ ones go to $\neww$ ones, and the $\rightt$ ones from them go to the right ones. Now let us prove that the $\leftt$ go to the $\leftt$.\par
\begin{lemma}
Let $\phi(z_3) = y_3$. Then $\phi(z_1)$ = $y_1$, where $y_1$ = $\leftt$ $(y_3)$ and $z_1$ = $\leftt$ $(z_3)$.
\end{lemma}
\ig{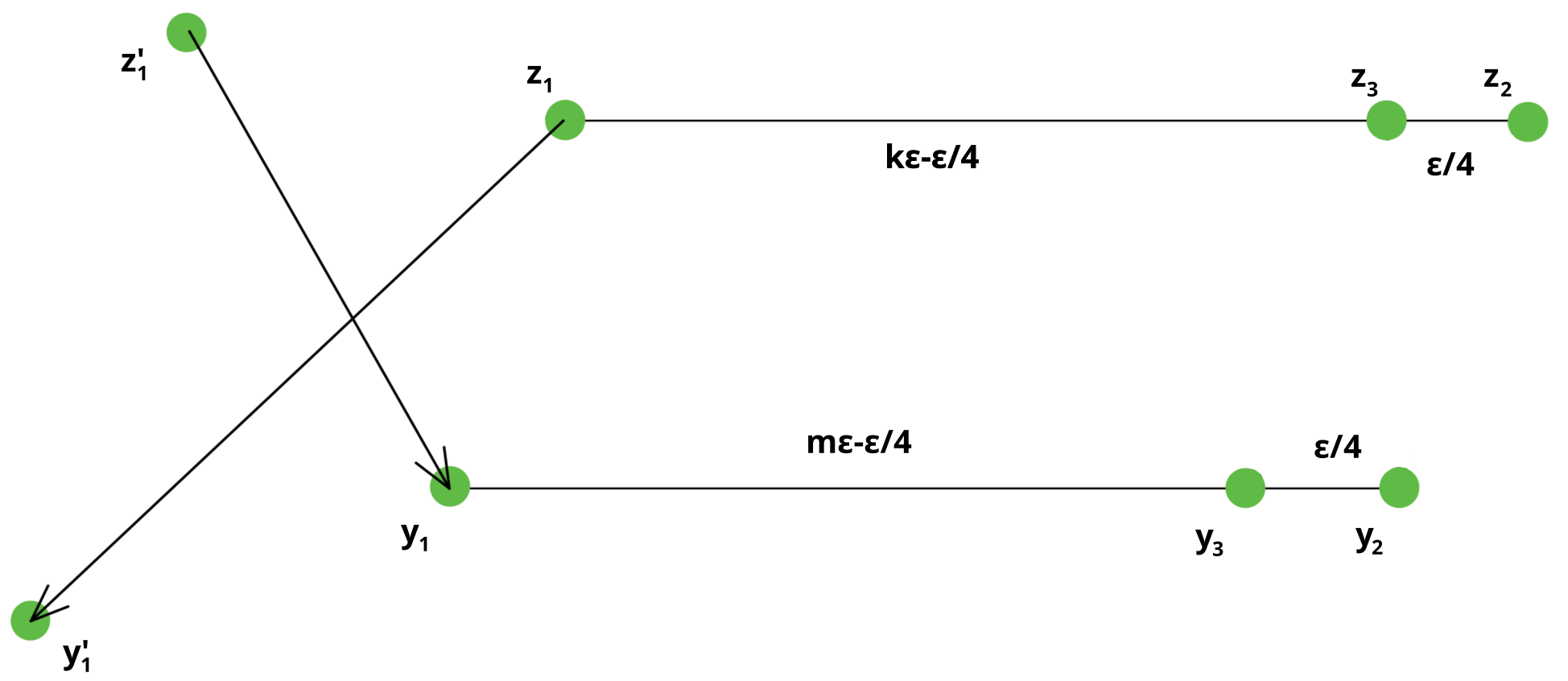}{0.2}{fig:pic4}{$\Leftt$ $z_1$ goes to $\farr$ $\leftt$($y_3$) = $y'_1$.}
\begin{proof}
Indeed, let $y_3$ = $\phi(z_3)$, $y_1$ = $\leftt$($y_3$) and $y_2$ = $\rightt$($y_3$). We need to prove that $\phi(z_1) = y_1$. Assume the contrary. Then $\phi(z_1) = y'_1$ = $\farleftt$($y_3$) by Lemma~\ref{sent:class}. Let $\phi^{-1}(y_1) = z'_1$. This point must be $\farleftt$($z_3$), since it cannot be $\leftt$ due to bijectivity. Let $d(z_1,z_2) = k_z\eps$, then $d(z_1,z_3) = k_z \eps - \eps/4$. Since $\phi$ is distance-preserving, $d\bigl(z_1, z_3 \bigr) = d\bigl(\phi(z_1),\phi(z_3)\bigr) = k_z \eps - \eps/ 4$. We put $d(y_1,y_2) = k_y \eps$. Then $d(y_1,y_3) = k_y\eps-\eps/4 < d(y_3,y'_1) = k_z\eps - \eps/4$ by Lemma~\ref{sent:blizh_left}, so $k_y<k_z$. Considering the inverse mapping, we obtain $k_z < k_y$. Contradiction. So $\phi(z_1) = y_1$.
\end{proof}
If $x \prec y$ $(x,y \in Z')$ then find $u$ which is $\neww$ for the pair $x,y$. Then $\phi(u)$ is some $\neww$, for which $\phi(x)$ and $\phi(y)$ are $\leftt$ and $\rightt$, respectively, so $\phi(x) \prec \phi(y) $, i.e. $\phi$ preserves the order. By Lemma~\ref{lemma:sohr_poradok}, the restriction of $\phi$ to $Z'$ is the identity mapping. By the construction of $Z$, the mapping $\phi$ is the identity mapping on the whole of $Z$.\par
Thus $\e(Z) \geq \eps/4$ and $s(Z) \geq \eps/4$. Note that $\dGH (X,Z) \leq \dGH(X,Z') + \dGH(Z,Z') \leq \eps/2 + \dH(Z,Z') \leq 3\eps /4$, since $Z'$ embeds in $Z$ via the identification.\par
Now let us use the Lemma~$\ref{lemma:xplusc}$ for arbitrary $c > 0$. Then $\e(Z + 2c) \geq \eps/4$, $\s(Z + 2c) \geq \eps/4 + 2c$, $\t(Z+ 2c) \geq 2c$. Putting $\delta = 3\eps/4$ and $U = Z + 2c$, we get $\dGH(X,U) \leq \delta + c$ by Lemma~\ref{lemma: dist_to_image}, $\ e(U) \geq \delta/3$, $\s(U) \geq \delta/3 + 2c$, $\t(U) \geq 2c$.
\end{proof}
Thus spaces in general position are everywhere dense in $\ca{GH}$.

\end{document}